\documentclass{amsart}
\usepackage{amssymb}
\input xy
\xyoption{all}

\newtheorem{innercustomthm}{Theorem}
\newenvironment{customthm}[1]
  {\renewcommand\theinnercustomthm{#1}\innercustomthm}
  {\endinnercustomthm}

\newtheorem{theorem}{Theorem}[section]
\newtheorem{proposition}[theorem]{Proposition}
\newtheorem{lemma}[theorem]{Lemma}
\newtheorem{corollary}[theorem]{Corollary}

\theoremstyle{definition}
\newtheorem{definition}[theorem]{Definition}

\theoremstyle{remark}
\newtheorem{remark}[theorem]{Remark}

\numberwithin{equation}{section}

\def \d {\delta}
\def \dd {\partial}
\def \D {\Delta}
\def \DD {\mathcal D}
\def \DDCF {{\mathcal D}\operatorname{-CF}_0}
\def \Der {\operatorname{Der}}
\def \C {\mathcal C_{\mathcal U}}

\title[On the DME for Poisson-Hopf algebras]{On the Dixmier-Moeglin equivalence for Poisson-Hopf algebras}

\author{St\'ephane Launois}
\address{St\'ephane Launois\\
University of Kent\\
School of Mathematics, Statistics, and Actuarial Science \\
Canterbury, Kent \ CT2 7NZ \\
UK}
\email{s.launois@kent.ac.uk}
\thanks{The work of SL was supported by EPSRC grant EP/N034449/1. }

\author{Omar Le\'on S\'anchez}
\address{Omar Le\'on S\'anchez\\
University of Manchester\\
School of Mathematics\\
Oxford Road \\
Manchester, M13 9PL.}
\email{omar.sanchez@manchester.ac.uk}

\date{\today}

\begin{document}

\begin{abstract}
We prove that the Poisson version of the Dixmier-Moeglin equivalence holds for cocommutative affine Poisson-Hopf algebras. This is a first step towards understanding the symplectic foliation and the representation theory of (cocommutative) affine Poisson-Hopf algebras. Our proof makes substantial use of the model theory of fields equipped with finitely many possibly noncommuting derivations. As an application, we show that the symmetric algebra of a finite dimensional Lie algebra, equipped with its natural Poisson structure, satisfies the Poisson Dixmier-Moeglin equivalence.
\end{abstract}

\maketitle

\tableofcontents

\section{Introduction}

We fix a field $k$ of characteristic zero, and recall that by an affine $k$-algebra one means a finitely generated one that is a domain. The aim of this note is to study the representation theory and the geometry of affine Poisson-Hopf $k$-algebras via methods from model theory of differential fields. For the reader unfamiliar with Poisson-Hopf algebras, let us mention at this point that these include symmetric algebras of finite-dimensional Lie algebras (endowed with their natural Poisson bracket, see Section \ref{sec:Poisson-Hopf}). For further examples, we refer the reader to \cite{KS} and \cite{LWZ}.

As it is often the case, classifying simple representations or symplectic leaves of Poisson(-Hopf) algebras is too wide a problem, and so we are approaching it by studying the so-called Poisson-primitive ideals. There are several equivalent ways to define these (prime Poisson) ideals. Given our representation theoretic and geometric motivation, we just give two equivalent definitions at this stage. Let $A=O(V)$ be an affine Poisson $k$-algebra (so that $V$ is a Poisson variety over $k$). The Poisson-primitive ideals of $A$ are the defining ideals of the Zariski closure of the symplectic leaves of $V$. Equivalently, they are precisely the annihilators of the simple Poisson $A$-modules. Thus, classifying Poisson-primitive ideals of a Poisson(-Hopf) algebra is a first step towards understanding both its symplectic foliation and its representation theory.

The goal of this paper is to provide a topological criterion to characterise Poisson-primitive ideals among prime Poisson ideals of a Poisson-Hopf algebra. More precisely, we prove that Poisson-primitive ideals of a cocommutative affine Poisson-Hopf algebra are exactly those prime ideals that are Poisson-locally closed. In \cite{BLLM}, the authors together with Bell and Moosa proved that Poisson-primitive ideals coincide with the so-called Poisson-rational ideals. So, combining this with our results here, we obtain that for affine cocommutative Poisson-Hopf algebras, the notions of Poisson-rational, Poisson-primitive and Poisson-locally closed coincide. The coincidence of these three notions is often referred to as the Poisson Dixmier-Moeglin equivalence (see Section \ref{sec:PDME}), so that we can state our main result as follows.

\begin{customthm}{1}
\label{thm:A}
 Any cocommutative affine Poisson-Hopf $k$-algebra satisfies the Poisson Dixmier-Moeglin equivalence. 
\end{customthm}

We note that the Poisson Dixmier-Moeglin equivalence does not always hold for a Poisson algebra. By \cite[1.7(i), 1.10]{Oh}, in an affine Poisson $k$-algebra $A$, we have that Poisson-locally closed implies Poisson-primitive, and Poisson-primitive implies Poisson-rational (this also follows from Proposition~\ref{imply} below). It was shown in \cite{BLLM}, that when $A$ has Krull-dimension at most three, then Poisson-rational implies Poisson-locally closed, and so the Poisson-DME holds in this case. However, in that same paper, counterexamples of Krull-dimension $d$ were build, for any $d\geq 4$. Hence, the main point of this note is to show that those (counter-)examples cannot admit a cocommutative Poisson-Hopf algebra structure; and that in fact in this case Poisson-rational does imply Poisson-locally closed. 

While in general one cannot remove the Hopf algebra assumption in Theorem~\ref{thm:A}, one natural question to ask at this point is: can we remove the cocommutative assumption? That is, does the Poisson Dixmier-Moeglin equivalence hold in any affine Poisson-Hopf algebra? While we currently do not have an answer, in Remark~\ref{suggest}(2) below we suggest how one could address this question. We note that in the differential-Hopf algebra context in a \emph{single derivation} the cocommutative assumption can indeed be removed, see \cite[Theorem 2.19]{BLM}, and also that Bell and Leung have asked a similar question in the noncommutative setting, see \cite[Conjecture 1.3]{BLe}.

As symmetric algebras of finite dimensional Lie algebras are examples of cocommutative affine Poisson-Hopf algebras, Theorem 1 applies to this family of Poisson $k$-algebras. Thus, we obtain (with very different methods) a Poisson analogue of the foundational result of Dixmier and Moeglin, later generalized by Irving and Small, that asserts that primitive ideals in the enveloping algebra $U(\frak{g})$ of a finite-dimensional Lie $k$-algebra $\frak{g}$ are precisely the locally closed prime ideals of  $U(\frak{g})$~\cite{DM}.

\begin{customthm}{2}
If $A$ is the symmetric algebra of a finite dimensional Lie $k$-algebra $\mathfrak g$, equipped with its natural Poisson bracket, then $A$ satisfies the Poisson Dixmier-Moeglin equivalence.
\end{customthm}

We expect that our results on the representation theory of Poisson-Hopf algebras will help us better understand representations of Hopf algebras in general. This is justified by the fact that connected Hopf algebras arise (in a rough sense) as deformations of Poisson-Hopf algebras \cite{Zhuang}.

It is important to note that a significant part of the proof of Theorem \ref{thm:A} makes substantial use of the model theory of differential fields, via the theory of algebraic $D$-varieties and $D$-groups (see Sections \ref{variso} and \ref{groupiso}). The novelty of this paper, compared to \cite{BLLM} or \cite{BLM} where the model theory of \emph{ordinary} differential fields was used, is that we work in the context of several possibly noncommuting derivations. While the model theory of differential fields with \emph{commuting} derivations has fruitfully been applied in other areas of mathematics (for instance, in differential Galois theory, see \cite{LSPillay}), to the authors knowledge this paper contains the first application of the model-theoretic properties of the theory of differential fields where no commutativity assumption is made among the derivations. We expect (and hope) that the ideas presented here will motivate the further use of these tools in new areas of algebra, and perhaps initiate the study of the model theory of Poisson rings.

\section{Some preliminaries}\label{pre}

Recall that for us $k$ denotes a field of characteristic zero, and that by an affine algebra we mean a finitely generated one that is a domain.

\subsection{Poisson-Hopf algebras}
\label{sec:Poisson-Hopf}

Recall that a Poisson $k$-algebra is a commutative $k$-algebra $A$ equipped with a Lie bracket $\{-,-\}$ such that
$$\{a,bc\}=\{a,b\}c+b\{a,c\}, \quad \text{ for all } a,b,c\in A.$$
In other words, for each $a\in A$, the map $\{a,-\}:A\to A$ is a derivation.

Given a Poisson $k$-algebra $(A,\{-,-\})$ the tensor algebra $A\otimes A$ can be naturally equipped with a Poisson $k$-algebra structure as follows; define
$$\{a\otimes b,a'\otimes b'\}=\{a,a'\}\otimes b\,b' +a\, a'\otimes \{b,b'\}$$
for $a,b,a'a,b'\in A$ and extend to all of $A\otimes A$ by $k$-linearity.

\begin{definition}
A {\em Poisson-Hopf $k$-algebra} $A$ is a Poisson $k$-algebra with the additional structure of a Hopf algebra such that the Poisson bracket $\{-,-\}$ commutes with coproduct $\Delta$; that is,
$$\Delta(\{a,b\})=\{\D(a),\D(b)\}$$
for all $a,b\in A$. 
\end{definition}

\begin{remark}\label{gen}
Let $(A,\{-,-\})$ be a Poisson algebra with a Hopf algebra structure. In order to prove that the Poisson bracket commutes with coproduct it suffices to check that if $G$ is a set of generators of $A$ (as a $k$-algebra) then
$$
\Delta(\{a,b\})=\{\D(a),\D(b)\}
$$
for all $a\in G$ and $b\in G$. Indeed, suppose $a\in G$ and $b\in A$, then $b=f(g_1,\dots,g_s)$ for some polynomial $f$ over $k$ and $g_i\in G$. Since $\{a,-\}$ and $\{\Delta(a),-\}$ are $k$-linear derivations, we get
\begin{align*}
\Delta(\{a,b\}) = & \Delta\left(\sum_{i=1}^s\frac{\partial f}{\partial t_i}(g_1,\dots,g_s)\cdot\{a,g_i\}\right) \\
 = & \sum_{i=1}^s\frac{\partial f}{\partial t_i}(\Delta(g_1),\dots,\Delta(g_s))\cdot\{\Delta(a),\Delta(g_i)\} \\
 = & \{\Delta(a),f(\Delta(g_1,\dots,\Delta(g_s))\} \\
 = & \{\Delta(a),\Delta(b)\}.
\end{align*}
The general case, when $a,b\in A$, follows from this in similar fashion; now writing $a$ as $h(g_1,\dots,g_s)$ and using the fact that $\{-,b\}$ and $\{-,\Delta(b)\}$ are $k$-linear derivations. 
\end{remark}

We now describe a general example of cocommutative Poisson-Hopf algebra. Namely, let $A$ be the symmetric algebra of a finite dimensional Lie $k$-algebra $\mathfrak g$. Such a (symmetric) algebra $A$ is of the form $k[x_1,\dots,x_s]$ where the $x_i$'s form a $k$-basis for the Lie algebra $\mathfrak g$. It is well known that $A$ becomes a Poisson algebra with Poisson bracket defined  by:
$$\{f,g\}:=\sum_{i<j} [x_i,x_j] \left( \frac{\partial f}{\partial x_i} \frac{\partial g}{\partial x_j} - \frac{\partial f}{\partial x_j}\frac{\partial g}{\partial x_i} \right) .$$
Moreover, we can equip $A$ with a Hopf algebra structure where the $x_i$'s are all primitive (so that $A$ is cocommutative). One can check that then the Poisson bracket commutes with the coproduct. Indeed, it follows from Remark~\ref{gen} that we only need to prove that 
$$
\Delta(\{x_i,x_j\})=\{\D(x_i),\D(x_j)\}
$$
for all $i,j$. Using the fact that every element of $\mathfrak g$ is primitive and that $\Delta$ is an algebra homomorphism, we get
\begin{align*}
\{\D(x_i),\D(x_j)\} = & \{1\otimes x_i + x_i \otimes 1,1 \otimes x_j +x_j \otimes 1 \} \\
 = & 1 \otimes \{x_i,x_j\} + \{x_i, x_j \} \otimes 1 \\
 = & 1 \otimes [x_i,x_j] + [x_i, x_j ] \otimes 1  \\
 = &\Delta ([x_i,x_j])\\
  = &\Delta (\{x_i,x_j\}) ~ \mbox{ for all }i,j.\\
\end{align*}
Thus $A$ is indeed a cocommutative affine Poisson-Hopf $k$-algebra. 

Further examples of (affine) Poisson-Hopf algebras are given by the coordinate ring of Poisson affine algebraic groups \cite{KS}. These examples are not, however, generally cocommutative, unless of course the algebraic group is abelian.

\subsection{The Poisson Dixmier-Moeglin equivalence}
\label{sec:PDME}

Let $(A,\{-,-\})$ be a Poisson $k$-algebra . An ideal $I$ of $A$ is a {\em Poisson ideal} if $\{A,I\}\subseteq I$. The {\em Poisson-center} of $A$ is defined as
$$Z_p(A)=\{a\in A: \{A,a\}=0\}.$$
Recall that when $A$ is a domain, there is a natural Poisson structure on $\operatorname{Frac} A$ (induced by the quotient rule of derivations). 

We denote by $Spec_P A$ the subspace of $Spec A$ consisting of Poisson ideals. The \emph{Poisson core} of an ideal $I$ of $A$ is defined as the largest Poisson ideal contained in $I$. A prime Poisson ideal $P$ of $A$ is said to be
\begin{itemize}
\item {\em Poisson-locally closed} if 
$$\bigcap_{P\subsetneq Q\in Spec_P A} Q\; \neq \; P$$
\item {\em Poisson-primitive} if $P$ is the Poisson core of a maximal ideal of $A$.
\item {\em Poisson-rational} if $Z_P(\operatorname{Frac}A/P)$ is an algebraic extension of $k$.
\end{itemize}

We say that a $A$ satisfies the \emph{Poisson Dixmier-Moeglin equivalence} (or Poisson-DME) if a prime Poisson ideal of $A$ is Poisson-locally closed iff it is Poisson-primitve iff it is Poisson-rational. Let us remark that, by \cite[Lemma 3.5]{BG}, the above definition of Poisson-primitive does correspond to the equivalent definitions given in the introduction.

\subsection{Differential algebras} 

We now briefly recall some facts about differential $k$-algebras that will be useful in subsequent sections. For any $k$-algebra $A$, we denote by $\Der_k(A)$ the $A$-vector space of $k$-linear derivations on $A$. 

\begin{remark}\label{ongenset}
Let $(A,\{-,-\})$ be a Poisson $k$-algebra and $G$ a set of generators of $A$. Recall that a Hamiltonian of $A$ is an element of $\Der_k(A)$ of the form $\{a,-\}$ for some $a\in A$. An easy computation shows that the $A$-vector subspace of $\Der_k(A)$ spanned by the Hamiltonians of $A$ is equal to
$$\operatorname{span}_A(\{a,-\}:\, a\in G).$$
\end{remark}

Due to the above remark, to check that an ideal in a Poisson algebra is a Poisson ideal one only needs to check that it is invariant under the Hamiltonians of a set of generators. More generally, we have

%Recall that $\Der_k(A)$ has a natural Lie algebra structure over $k$ by taking commutators. 

\begin{lemma}\label{onideals}
Let $(A,\{-,-\})$ be a Poisson $k$-algebra. If $\DD \subseteq \Der_k(A)$ is such that 
$$\operatorname{span}_A\DD=\operatorname{span}_A(\{a,-\}:\, a\in A),$$
then 
\begin{enumerate}
\item an ideal $I$ of $A$ is Poisson iff it is a $\DD$-ideal (i.e., invariant under $\DD$), and
\item the Poisson-center $Z_{P}(A)$ equals the $\DD$-constants of $A$ (i.e., $\displaystyle \cap_{\d\in \DD}\operatorname{ker}\d$). Furthermore, if $A$ is a domain, 
$$Z_P(\operatorname{Frac}A)=\DD\text{-constants}(\operatorname{Frac} A).$$
\end{enumerate}
\end{lemma}

\begin{proof}
This is an easy exercise. We leave the details to the reader.
\end{proof}

Suppose $A$ is a commutative $k$-algebra equipped with a family of $k$-linear derivations $\DD$. We denote by $Spec_\DD A$ the subspace of $Spec A$ consisting of $\DD$-ideals. The \emph{$\DD$-core} of an ideal $I$ of $A$ is defined as the largest $\DD$-ideal contained in $I$. A prime $\DD$-ideal $P$ of $A$ is said to be
\begin{itemize}
\item {\em $\DD$-locally closed} if 
$$\bigcap_{P\subsetneq Q\in Spec_\DD A} Q\; \neq\;  P$$
\item {\em $\DD$-primitive} if $P$ is the $\DD$-core of a maximal ideal of $A$.
\item {\em $\DD$-rational} if the $\DD$-constants of $\operatorname{Frac}A/P$ is an algebraic extension of $k$.
\end{itemize}

We say that a $A$ satisfies the \emph{$\DD$-Dixmier-Moeglin equivalence} (or $\DD$-DME) if a prime $\DD$-ideal of $A$ is $\DD$-locally closed iff it is $\DD$-primitve iff it is $\DD$-rational. 

The following is an easy consequence of Lemma~\ref{onideals}.

\begin{corollary}\label{iff}
Let $(A,\{-,-\})$ be a Poisson $k$-algebra. If $\DD \subseteq \Der_k(A)$ is such that 
$$\operatorname{span}_A\DD=\operatorname{span}_A(\{a,-\}:\, a\in A),$$
then a prime ideal $P$ of $A$ is 
\begin{enumerate}
\item Poisson-locally closed iff it is $\DD$-locally closed,
\item Poisson-primitive iff it is $\DD$-primitive, and
\item Poisson -rational iff it is $\DD$-rational.
\end{enumerate}
Consequently, $(A,\{-,-\})$ satisfies the Poisson-DME iff it satisfies the $\DD$-DME.
\end{corollary}

Given a commutative Hopf $k$-algebra $A$ equipped with a family of $k$-linear derivations $\DD$, we say that $A$ is a \emph{differential-Hopf algebra} if each derivation commutes with coproduct; that is,
$$\d(\D(a))=\D(\d a)$$ 
for all $a\in A$ and $\d\in \DD$. Here recall that the derivations $\DD$ naturally lift to $A\otimes A$ as follows;
$$\d(a\otimes b)=\d a\otimes b+ a\otimes \d b$$
for all $a,b\in A$ and extend by $k$-linearity to all of $A\otimes A$.

\begin{remark} \
\begin{enumerate}
\item As we did in Remark \ref{gen}, one can check that, in a commutative Hopf $k$-algebra $A$ equipped with $k$-linear derivations $\DD$, the derivations $\DD$ commute with coproduct if and only if $\d(\D(a))=\D(\d a)$ for all $\d\in \DD$ and $a$ varying in a set of generators of $A$.
\item Suppose $(A,\{-,-\})$ is an affine Poisson-Hopf $k$-algebra, we do not know if there is $\DD \subseteq \Der_k(A)$ with
$$\operatorname{span}_A\DD=\operatorname{span}_A(\{a,-\}:\, a\in A)$$
and such that $(A,\DD)$ is a differential-Hopf algebra. Nonetheless, in Proposition \ref{key}, we prove that it is possible to find such $\DD$ in the case when $k$ is algebraically closed and $A$ is cocommutative.
\end{enumerate}
\end{remark}

\section{On affine $D$-varieties and isotriviality}\label{variso}

In this section we present the basics of the theory of affine algebraic $D$-varieties in the context of finitely many (possibly noncommuting) derivations, together with the notions of isotriviality and compound isotriviality. It is worth noting that the theory of $D$-varieties in the context of commuting derivations appears in \cite{Buium}.

We make, somewhat freely (specially compared to \cite{BLLM,BLM}), use of basic model-theoretic terminology, for which \cite{Marker} should suffice. The appropriate model-theoretic context here is that of fields equipped with finitely many (possibly noncommuting) derivations. Fix a positive integer $m$. We work in the first-order language of differential rings equipped with $m$ derivations
$$\mathcal L_{diff}=\mathcal L_{rings}\cup\{\d_1,\dots,\d_m\}.$$
One can easily axiomatize the class of differential fields $(K,\d_1,\dots,\d_m)$ of characteristic zero (where no commutativity assumption is made among the derivations). Such a differential field is called \emph{existentially closed} if any quantifier-free $\mathcal L_{diff}$-formula with a realisation in a differential field extension of $K$ already has a realisation in $K$. Note that such differential fields exist (by a standard Zorn's lemma and chain construction argument). 

It turns out that there is a first-order axiomatization of the class of existentially closed differential fields of characteristic zero. This is a consequence of the general results in \cite{MS}, and following their notation we denote this theory by $\DDCF$. In that same paper, the authors established that this is a complete stable theory with quantifier elimination and elimination of imaginaries. Moreover, they showed that given $(\mathcal U, \d_1,\dots,\d_m)\models \DDCF$, if $K$ is a differential subfield then $K=dcl^{\mathcal U}(K)$, and if additionally $K$ is algebraically closed then $K=acl^{\mathcal U}(K)$.

We now fix a sufficiently large saturated model $(\mathcal U,\DD= \{\d_1,\dots,\d_m\})\models \DDCF$. This means that given a small (i.e., $|K|<|\mathcal U|$) differential subfield $K$ of $\mathcal U$, and a (possibly infinite) collection $\Sigma$ of $\mathcal L_{diff}$-formulas with parameters from $K$, if every finite subcollection of $\Sigma$ is satisfiable in $\mathcal U$ then so is all of $\Sigma$. One of the most important definable subsets of $\mathcal U$ is its subfield of constants, which is defined as
$$\mathcal{C}_{\mathcal U}=\bigcap_{i=1}^m \operatorname{ker} \d_i.$$
%By a \emph{constructible set} over $K$ we mean a subset of $\mathcal U^n$, for some $n$, that is the solution of a basic (equivalently atomic) $\mathcal L_{diff}$-formula over $K$ or finite Boolean combination of such. 
A subset of $\mathcal U^n$ that is an arbitrary (possibly infinite) intersection of definable sets will be called a \emph{type-definable} set. 

\begin{remark}\label{ondef}
The field $\C$ is algebraically closed and it is \emph{purely stably embedded}. This means that any subset of $\C^n$ that is definable in the $\mathcal L_{diff}$-structure $\mathcal U$, over some differential subfield $K$, is actually definable in the $\mathcal L_{rings}$-structure $\C$ over $\mathcal C_{K}$. In particular, by $\omega$-stability of the theory of algebraically closed fields, if $G\subseteq \C^n$ is a type-definable group over $K$ in the differential structure $(\mathcal U, \d_1,\dots,\d_m)$, then $G$ is definable over $C_K$ in the pure-field structure $\C$.
\end{remark}

We now discuss affine algebraic $D$-varieties. Fix a (small) differential subfield $K<\mathcal U$. We say that a Zariski closed set $V\subseteq \mathcal U^n$ defined over $K$ is a $D$-variety if its coordinate ring $K[V]$ is equipped with a family of $m$ derivations $\bar \dd=\{\dd_1,\dots,\dd _m\}$ such that $\dd_i$ extends $\d_i|_K$.

We now wish to give a more algebro-geometric characterization of affine $D$-varieties. Given a Zariski closed $V\subseteq \mathcal U^n$ over $K$ and $\d\in \DD$, the \emph{$\d$-prolongation of $V$} is the Zariski-closed set $\tau_\d\subseteq \mathcal U^{2n}$ defined by the equations
$$f(\bar x)=0 \quad \text{ and }\quad \sum_{i=1}^n\frac{\partial f}{\partial x_i}(\bar x)\cdot y_i+f^{\d}(\bar x)=0$$
for all $f\in \mathcal I(V/K):=\{f\in K[\bar x]: f(V)=0\}$, where $\bar x=(x_1,\dots,x_n)$ and $f^\d\in K[\bar x]$ is obtained by applying $\d$ to the coefficient of $f$. It is easy to check that it suffices to vary $f$ in a family of generators of the ideal $\mathcal I(V/K)$. Consequently, if $V$ is defined over a field $k$ of constants (i.e., $k< \C$), then $\tau_\d V$ is nothing more that the tangent bundle $TV$ of $V$.

More generally, the $\DD$-prolongation of $V$, denoted by $\tau_{\DD} V \subset \mathcal U^{n(m+1)}$, is defined as the fibred-product
$$\tau_{\DD}V=\tau_{\d_1} V\times_V\cdots \times_V \tau_{\d_m} V.$$
Note that $\tau_\DD V$ comes equipped with a canonical projection map $\tau_\DD V \to V$. The $\DD$-prolongation has the characteristic property that for any $a\in V$ we have that
$$(a,\d_1 a,\dots ,\d_m a)\in \tau_\DD V;$$
in other words, the map $\bar x\mapsto (\bar x,\d_1 \bar x,\dots,\d_m\bar x)$ defines a differential regular section of $\tau_\DD V\to V$.

If $V$ is defined over a field $k$ of constants, then $\tau_\DD V$ equals the $m$-fold fibred-product of $TV$ and hence it comes with a canonical (algebraic) section, the zero section $s_0:V\to \tau_\DD V$. On the other hand, for arbitrary $V$ defined over $K$, the existence of an algebraic regular section $s:V\to \tau_\DD V$ over $K$ turns out to be equivalent to a $D$-variety structure on $V$. Indeed, if $\bar\partial=\{\dd_1,\dots,\dd_m\}$ are derivations on the coordinate ring $K[V]$ extending those on $K$, and we let $\bar z=(z_1,\dots,z_n)$ be its coordinate functions and set
$$s(\bar z)=(\bar z,\dd_1(\bar z),\dots,\dd_m(\bar z))$$
where $\dd_i(\bar z)=(\dd_i(z_1),\dots,\dd_i(z_n))$, then it is not hard to check (using the fact that the $\dd_i$'s are derivations) that this $s$ yields a section of $\tau_\DD V\to V$ which is regular and over $K$. On the other hand, any such section
$$s=(\operatorname{Id}, s_1,\dots,s_m)$$
corresponds to the derivations on $K[V]$ induced by setting $\dd_i(\bar z)=s_i(\bar z)$. 

From now on, we will usually refer to a $D$-variety as a pair $(V,s)$ where $V$ is an affine algebraic variety (viewed as a Zariski closed subset of $\mathcal U^n$) and $s$ is a section of $\tau_\DD V\to V$. A point $a \in V$ will be called a $D$-point of $V$ if 
$$s(a)=(a,\d_1(a),\dots,\d_m(a)).$$
We will denote the set of $D$-points of $V$ by $(V,s)^\#$. Note that this is an example of a definable set in the structure $(\mathcal U, \d_1,\dots,\d_m)$ and in fact this will be the main source of such examples for us. Also, note that if $V$ is defined over a field of constants and $s=s_0$ (the zero section), then $(V,s)^\#$ equals $V(\C)$, the $\C$-points of $V$.

A Zariski closed subset $W$ of a $D$-variety $(V,s)$ is said to be a \emph{$D$-subvariety} if $s(W)\subseteq \tau_\DD W$; of course, in this case $(W,s|_W)$ will be a $D$-variety. Also, a regular map $f:V\to W$ between $D$-varieties $(V,s)$ and $(W,t)$ is said to be a $D$-morphism if $f$ maps $D$-points to $D$-points.

\begin{remark}
It is easy to check that $W\subset V$ is a $D$-subvariety iff the ideal $\mathcal I(W/K)\subset K[V]$ is a $\bar\partial$-ideal. Also, a regular map between $D$-varieties $f:V\to W$ is a $D$-morphism iff the pull-back $f^*:K[W]\to K[V]$ is a $\bar\partial$-ring homomorphism.
\end{remark}

\begin{lemma}\label{basic}
Let $(V,s)$ and $(W,t)$ be affine algebraic $D$-varieties over $K$ and $f$ a $D$-morphism between them. Then,
\begin{enumerate}
\item each $K$-irreducible component of $V$ is a $D$-subvariety,  
\item the set of $D$-points of $V$ is Zariski-dense in $V$,
\item if $X$ is a $D$-subvariety of $W$, then $f^{-1}(X)$ is a $D$-subvariety of $V$, and
\item if $Y$ is a $D$-subvariety of $V$, then the Zariski closure of $f(Y)$ is a $D$-subvariety of $W$.
\end{enumerate}
\end{lemma}
\begin{proof}
(1) Let $W$ be a $K$-irreducible component of $V$ and $a$ a Zariski $K$-generic point of $W$. It suffices to show that $s(a)\in \tau_\DD W$. As $a$ is not contained in any of the other $K$-irreducible components of $V$, from the nature of the equations defining $\tau_\DD W$, we get that $\tau_\DD W_a$ (the fibre above $a$) coincides with $\tau_\DD V_a$. The claim now follows.

(2) By (1), we may assume that $V$ is $K$-irreducible. Let $W$ be a proper Zariski closed subset of $V$ and $a$ a Zariski $K$-generic point of $V$ (hence $a\notin W$). Let $b=(b_1,\dots,b_m)\in \mathcal U^{nm}$ be such that $(a,b)=s(a)$, since $b\in \tau_\DD V_ a$, the equations of $\tau_\DD V$ yield that there are derivations $\dd_i: K(a,b)\to K(a,b)$, for $i=1,\dots,m$, such that $\dd_i|_K=\d_i|_K$ and $\dd_i(a)= b_i$ (see for instance \cite[Chapter~7, \S5]{Lang}). By saturation of $\mathcal U$, there is a differential field embedding $(K(a,b),\bar \dd) \to (\mathcal U,\DD)$ fixing $K$. So there is point $(a',b')$ in $\mathcal U^{n(m+1)}$, namely the image of $(a,b)$, such that $a'\in V\setminus W$ and 
$$s(a')=(a',b')=(a,\d_1 a',\dots,\d_m a').$$
That is, $a'$ is a $D$-point of $V$ not in $W$. Thus, the set of $D$-points is dense in $V$.

(3) To prove that $f^{-1}(X)$ is a $D$-subvariety of $V$ it suffices to show that 
$$\mathcal I(f^{-1}(X)/K)\subseteq K[V]$$ 
is a $\DD$-ideal. Recall that this ideal is given as the radical ideal generated by $f^*(\mathcal I(X/K))$. As radical ideals of $\DD$-ideals are again $\DD$-ideals (see \cite[Lemma 1.8]{Kap}), it suffices to show that the ideal generated by $f^*(\mathcal I(X/K))$ in $K[V]$ is a $\DD$-ideal. Let $\d\in \DD$, then for an expression of the form $\sum_i g_i f^*(h_i)$, with $g_i\in K[V]$ and $h_i\in \mathcal I(X/K)$, we have
$$\d\left(\sum_i g_i f^{*}(h_i)\right)=\sum_i\; \d (g_i)f^*(h_i) +g_i f^*(\d h_i),$$
where we have used that $f$ is a $D$-morphism (so $f^*$ commute with $\d$). Since $X$ is a $D$-subvariety of $W$, $\d h_i\in \mathcal I(X/K)$, and so the above term is in the ideal generated by $f^*(\mathcal I(X/K))$, as desired.

(4) To prove that $Z$, the Zariski closure of $f(Y)$, is a $D$-subvariety of $W$, it suffices to show that each $K$-irreducible component of $Z$ is such. Thus, we assume that $Z$ is $K$-irreducible. Let $a$ be a Zariski $K$-generic $D$-point of an $K$-irreducible component of $Y$ that maps dominantly onto $Z$. Then, $b=f(a)$ is a Zariski $K$-generic of $Z$, and since $f$ is a $D$-morphism we have that $b$ is a $D$-point of $W$. Thus, 
$$s(b)=(b,\d_1 b, \dots,\d_m b)\in \tau_\DD Z,$$
and, by Zariski genericity of $b$, we must have $s(Z)\subseteq \tau_{\DD}Z$, as desired.
\end{proof}

From the above lemma we see that if $(V,s)$ is $K$-irreducible (meaning that $V$ is $K$-irreducible), then it contains a Zariski $K$-generic $D$-point. Indeed, if this were not the case, by saturation of $\mathcal U$ there would be a finite collection of proper Zariski closed subsets of $V$ defined over $K$ that contain all $D$-points of $V$. But as $V$ is $K$-irreducible this finite collection does not cover all of $V$ and hence we contradict part (2) of the lemma. 

Note that if $a$ is a Zariski $K$-generic $D$-point of $V$ (assuming $V$ is $K$-irreducible), then the function field $K(V)\cong K(a)$ equipped with the derivations $\bar \dd$ is a differential subfield of $(\mathcal U,\DD)$. Thus, from now on, we will assume that 
$$(K[V],\bar \dd)\leq (K(V),\bar \dd)<(\mathcal U,\DD),$$
and, moreover, we identify $\bar \dd$ with $\DD$.

\begin{proposition}\label{imply}
Let $k<\C$ and $(V,s)$ be a $k$-irreducible affine algebraic $D$-variety. Then, for any prime $\DD$-ideal $P$ of $k[V]$ we have
$$\DD\text{-locally closed}\implies \DD\text{-primitive}\implies \DD\text{-rational}.$$
Furthermore, suppose $V$ is geometrically irreducible (i.e., $k^{alg}$-irreducible), if in $k^{alg}[V]$ a prime $\DD$-ideal is $\DD$-rational only if it is $\DD$-locally closed, then the same holds in $k[V]$.
\end{proposition}
\begin{proof}
By passing to the quotient $k[V]/P$ we may assume that $P=(0)$. 

Now assume $(0)$ is $\DD$-locally closed. This means that 
$$\bigcap_{Q\in Spec_\DD k[V]\setminus (0)} Q\neq (0).$$
At the level of $D$-subvarieties of $V$ (recall that each such $Q$ corresponds to a proper $D$-subvariety of $V$), this is equivalent to $V$ having a proper $D$-subvariety $W$ over $k$ that contains all such. Take a point $a\in V\setminus W(k^{alg})$. Then $\mathcal I(a/k)\subset k[V]$ is a maximal ideal and any $\DD$-ideal inside it must be zero (as $a\notin W$). This shows that $(0)$ is the $\DD$-core of a maximal ideal of $k[V]$; in other words, $(0)$ is $\DD$-primitive.

On the other hand, assume $(0)$ is $\DD$-primitive. That is, there is a maximal ideal $\mathfrak m$ with $\DD$-core $(0)$. Let $a\in V(k^{alg})$ be such that $\mathfrak m=\mathcal I(a/k)$. Now let $f$ be a $\DD$-constant of $k(V)$. We must show that $f$ algebraic over $k$. We first show that $a$ is not in the singular locus of $f$. Towards a contradiction suppose it is. Then, for every representation $f=\frac{p}{q}$ we have $q(a)=0$; that is, $q\in \mathfrak m$. Since $f$ is a $\DD$-constant, we have that for each $\d\in \DD$
$$0=\d f=\frac{\d p\cdot q-p\cdot \d q}{q^2}$$
and so, either $\d q=0$, or $\frac{\d p}{\d q}=\frac{p}{q}=f$ in which case $\d q(a)=0$. In any case, $\d q\in \mathfrak m$. Repeating this process we obtain that $\d'\d q\in\mathfrak m$ for any $\d'\in \DD$, and so on, hence we get that the $\DD$-ideal generated by $q$ is contained in $\mathfrak m$. This contradicts the fact that the $\DD$-core of $\mathfrak m$ is $(0)$, and so $f$ is defined at $a$. Write $f=\frac{p}{q}$ where $q(a)\neq 0$. Now let $h\in k[t]$ be the minimal polynomial of $f(a)\in k^{alg}$. There is a sufficiently large integer $s$ such that if we set
$$r=q^s\cdot (h\circ f)$$
then $r\in k[V]$; and, since $r(a)=q^s(a)h(f(a))=0$, we also have that $r\in \mathfrak m$. Let $\d\in \DD$, since $\d(h\circ f)=(h'\circ f)\cdot \d f=0$, we get that $\d r=\d(q^s)\cdot(h\circ f)$; and so $\d r(a)=0$, implying that $\d r\in \mathfrak m$. Repeating this process we obtain that $\d'\d r\in\mathfrak m$ for any $\d'\in \DD$, and so on, hence we get that the $\DD$-ideal generated by $r$ is contained in $\mathfrak m$. By the choice of $\mathfrak m$, $r$ must be zero. This implies that $h(f)=0$ and so $f\in k^{alg}$. This shows $\DD$-rationality of $(0)$.  

For the 'furthermore' clause, suppose $V$ is geometrically irreducible and that a prime $\DD$-ideal of $k^{alg}[V]$ is $\DD$-rational only if it is $\DD$-locally closed. Let $P$ be a prime $\DD$-rational ideal of $k[V]$. We must show that $P$ is $\DD$-locally closed. Let $W$ be the $k$-irreducible $D$-subvariety of $V$ that corresponds to $P$. Also, let $Y$ be one of the $k^{alg}$-irreducible components of $W$. By Lemma \ref{basic}(1), $Y$ is a $D$-subvariety of $V$ (over $k^{alg}$). Let $b$ be a Zariski $k^{alg}$-generic $D$-point of $Y$; then $b$ is a Zariski $k$-generic point $D$-point of $W$. Since
$$k^{alg}(Y)= k^{alg}(b)\subseteq k(b)^{alg}$$
and the $\DD$-constants of $k(b)$ are algebraic over $k$ (by $\DD$-rationality of $P$), we get that the $\DD$-constants of $k^{alg}(Y)$ are precisely $k^{alg}$. In other words, the $\DD$-ideal $\mathcal I(Y/k^{alg})$ of $k^{alg}[V]$ is $\DD$-rational; by our assumption, this ideal is $\DD$-locally closed. That is, $Y$ contains a proper $D$-subvariety $Y'$ over $k^{alg}$ that contains all such. Letting $W'$ be the Zariski $k$-closure of $Y'$, we obtain a proper $D$-subvariety of $W$ over $k$ containing all such; equivalently, $P$ is $\DD$-locally closed. 
\end{proof}

We will need one more piece of model-theoretic terminology. We denote by $Aut_\DD(\mathcal U/K)$ the differential automorphisms of $\mathcal U$ fixing $K$. Given a tuple $a\in \mathcal U^n$, we define the (complete) type of $a$ over $K$ as the orbit of $a$ under the action of $Aut_\DD(\mathcal U/K)$ on $\mathcal  U^n$; in other words, 
$$tp(a/K):=\{b\in \mathcal U^n: \sigma(a)=b \text{ for some } \sigma\in Aut_\DD(\mathcal U/K)\}.$$
In model-theoretic parlance, we are identifying a type with its set of realizations in $\mathcal U$. A type $tp(a/K)$ is always given as an (infinite) intersection of $\mathcal L_{diff}$-definable sets in $\mathcal U$; namely, the intersections of all definable sets over $K$ containing $a$. We will call the set defined by $tp(a/K)$, in $\mathcal U^n$, the set of its realisations. We say that $tp(a/K)$ is \emph{isolated} if its set of realizations is a definable set (it will necessarily be definable over $K$).

\begin{remark}\label{gentype}
If $(V,s)$ is a $K$-irreducible $D$-variety and $a$ is a Zariski $K$-generic $D$-point, then the type $tp(a/K)$ is precisely the set of all Zariski $K$-generic $D$-points of $V$. Indeed, if $b$ is another Zariski $K$-generic $D$-point, then, by saturation of $\mathcal U$, there is a field automorphism $\sigma$ of $\mathcal U$ such that $b=\sigma(a)$. But since both, $a$ and $b$, are $D$-points, $\sigma$ is in fact a \emph{differential} homomorphism, and so $\sigma \in Aut_\DD(\mathcal U/K)$. The other implication is obvious.
\end{remark}

\begin{proposition}\label{isoloc}
Let $k$ be a subfield of $\C$. Let $(V,s)$ be a $k$-irreducible $D$-variety, and $a$ a Zariski $k$-generic $D$-point of $V$. Then $tp(a/k)$ is isolated if and only if $(0)$ is a $\DD$-locally closed $\DD$-ideal of $k[V]$. 
\end{proposition}
\begin{proof}
Suppose $tp(a/k)$ is isolated. Then its set of realizations is definable over $k$. By quantifier elimination and the fact that all such realizations are $D$-points of $V$, this definable set must the form $( V\setminus W)\cap (V,s)^\#$ where $W$ is a proper Zariski closed subset of $V$ defined over $k$. Since, by Remark \ref{gentype}, the realisations of $tp(a/K)$ is precisely the set of all Zariski $k$-generic $D$-points of $V$, all $D$-subvarieties of $V$ defined over $k$ must be contained in $W$. At the level of $\DD$-ideals of $k[V]$, this is equivalent to saying that all the nonzero $\DD$-ideals contain $\mathcal I(W/k)$. This shows that 
$$\bigcap_{Q\in Spec_\DD k[V]\setminus (0)} Q\neq (0),$$
and so $(0)$ is $\DD$-locally closed.

On the other hand, assume $(0)$ is $\DD$-locally closed. Let $X$ be the proper $D$-subvariety of $V$ corresponding to the $\DD$-ideal $\cap_{Q\in Spec_\DD k[V]\setminus (0)} Q$. Note that $X$ contains all proper $D$-subvarieties of $V$ defined over $k$. If $a\in (V\setminus X)$ is a $D$-point, then it must a Zariski $k$-generic of $V$ (otherwise, its Zariski $k$-locus would yield a proper $D$-subvariety over $k$ not contained in $X$). This shows that the set
\begin{equation}\label{isol}
(V\setminus X)\cap (V,s)^\#
\end{equation}
coincides with the set of Zariski $k$-generic $D$-points of $V$. By Remark \ref{gentype}, this latter set is precisely $tp(a/K)$, so this type is given by the formula \eqref{isol} and therefore it is isolated.
\end{proof}

\begin{definition}[c.f. \cite{BLM}]\label{compound}
Let $(V,s)$ be a $K$-irreducible affine algebraic $D$-variety. 
\begin{enumerate}
\item We say that $(V,s)$ is isotrivial if there is a differential field $F<\mathcal U$ extension of $K$ and an injective $D$-morphism over $F$ from $(V,s)$ to $(W,s_0)$, where $W$ is defined over $\mathcal C_F$ and $s_0$ is the zero section. 
\item We say that $(V,s)$ is compound isotrivial in $\ell$-steps if there is an sequence of $K$-irreducible $D$-varieties $(V_i,s_i)$, $i=1,\dots,\ell$, and dominant $D$-morphisms over $K$
$$
\xymatrix{
V=V_{\ell}\ar[r]^{ \ \ \ f_\ell}&V_{\ell-1}\ar[r]^{f_{\ell -1}}&\cdots\ar[r]^{f_2}&V_{1}\ar[r]^{f_{1} \ \ \ }&V_0=0}
$$
such that for each $i=0,1,\dots,\ell-1$, if $a$ is a Zariski $K$-generic $D$-point of $V_{i}$, then $f_{i+1}^{-1}(a)$ is isotrivial. Here note that $f_{i+1}^{-1}(a)$ is a $K(a)$-irreducible $D$-subvariety of $V_{i+1}$ (by Lemma \ref{basic}(3)).
\end{enumerate}
\end{definition}

We now prove one of the key results of the paper. 

\begin{theorem}\label{ratiso}
Let $V$ be a $K$-irreducible affine compound isotrivial $D$-variety. If $\mathcal C_{K(V)}$, the $\DD$-constants of $K(V)$, is algebraic over $\mathcal C_K$, then the type of a Zariski $K$-generic $D$-point of $V$ is isolated.
\end{theorem}
\begin{proof}
We let $a$ be a Zariski $K$-generic $D$-point of $V$. We must show that $tp(a/K)$ is isolated. Suppose $V$ is compound isotrivial in $\ell$-steps. We proceed by induction on $\ell$. 

For the base case, $\ell=1$, $V$ must be isotrivial. By definition, there is a definable (with possibly additional parameters) injective map from the $D$-points of $V$ to $\C$. In model theoretic terms this means that the type $tp(a/K)$ is \emph{internal} to $\C$. This in turn implies that $tp(a/K^{alg})$ is internal to $\C$ as well. On the other hand, the condition that $\mathcal C_{K(a)}=\mathcal C_{K(V)}\subset \mathcal C_K^{alg}$ translates in model-theoretic terms to the type $tp(a/K^{alg})$ being \emph{weakly orthogonal} to $\C$. Indeed, to see this, one must show that $K^{alg}(a)\cap K^{alg}(\C)$ is contained in $K^{alg}$. Taking $d$ in the intersection we get $d=h(c)$ for some polynomial $h$ over $K^{alg}$ and tuple $c$ from $\C$. If $X$ is the set of tuples $x$ from $\C$ such that $d=h(x)$, then, by Remark \ref{ondef}, $X$ is $\mathcal L_{rings}$-definable in $\C$ over $\mathcal C_{K^{alg}(a)}$. Hence, there is a tuple $c'$ from $\mathcal C_{K(a)}^{alg}\subset \mathcal C_K^{alg} $ that satisfies $X$; and so $d=h(c')\in K^{alg}$, as desired.  

It is a well known model-theoretic fact on binding groups of automorphisms (for a proof see \cite[Appendix B]{Hrushovski}) that the above two conditions (internality and weak orthogonality) imply the existence of a type-definable group $G\subseteq \C^n$ over $K^{alg}$ that acts definably (over $K^{alg}$) and transitively on $tp(a/K^{alg})$. By Remark~\ref{ondef}, $G$ must be definable (over $\mathcal C_{K^{alg}}$), and so the realisations of $tp(a/K^{alg})$ form a definable set; in other words, $tp(a/K^{alg})$ is isolated. Since the type of any tuple of $K^{alg}$ over $K$ isolated, we must have that $tp(a/K)$ is isolated as well, as desired. 

Now assume $\ell>1$, and suppose we have 
$$
\xymatrix{
V=V_{\ell}\ar[r]^{ \ \ \ f_\ell}&V_{\ell-1}\ar[r]^{f_{\ell -1}}&\cdots\ar[r]^{f_2}&V_{1}\ar[r]^{f_{1} \ \ \ }&V_0=0}
$$ 
as in Definition \ref{compound}(2). Let $b=f_\ell(a)$, then $b$ is a Zariski $K$-generic $D$-point of $V_{\ell-1}$. Since $\mathcal C_{K(b)}\subseteq \mathcal C_{K(a)}$, we have that $\mathcal C_{K(b)}$ is algebraic over $\mathcal C_K$; and so, by induction, $tp(b/K)$ is isolated. We now claim that $tp(a/K(b))$ is also isolated. Indeed, $a$ is a Zariski $K(b)$-generic $D$-point of $W:=f^{-1}(b)$. By definition of compound isotriviality, $W$ is isotrivial, and so by the base case ($\ell=1$) the type of $a$ over $K(b)$ is isolated. The result now follows from the fact that both types, $tp(b/K)$ and $tp(a/K(b))$, are isolated (this is an easy exercise but a proof appears in \cite[Lemma 4.2.21]{Marker}). 
\end{proof}

\begin{corollary}\label{impcor}
Let $k$ be a subfield of $\C$ and let $(V,s)$ be a $k$-irreducible affine compound isotrivial $D$-variety. If $(0)$ is a $\DD$-rational ideal of $k[V]$, then it is also $\DD$-locally closed. 
\end{corollary}
\begin{proof}
The fact that $(0)$ is $\DD$-rational translates to $\mathcal C_{k(V)}$ being algebraic over $k=\mathcal C_k$ (this equality holds since $k< \C)$. Now, by Theorem~\ref{ratiso}, the type of a Zariski $k$-generic $D$-point of $V$ is isolated; which, by Proposition~\ref{isoloc}, implies that $(0)$ is a $\DD$-locally closed ideal of $k[V]$, as desired.
\end{proof}

\section{The $\DD$-DME for commutative affine $D$-groups over constants}\label{groupiso}

In this section we discuss affine algebraic $D$-groups, and show that the connected commutative ones defined over the constants are compound isotrivial in 2-steps. We carry on the notation from the previous section; in particular, $(\mathcal U,\DD=\{\d_1,\dots,\d_m\})$ is a sufficiently large saturated model of $\DDCF$ and all base differential fields, $k$ or $K$, of parameters are assumed to be small (i.e., of cardinality less than that of $\mathcal U$).

Given an affine algebraic group $G$ over $K$, just as the tangent bundle $TG$ of $G$ has the structure of an algebraic group, the $\DD$-prolongation of $G$ also has a canonical structure of an algebraic group (over $K$). 

\begin{definition}
An affine algebraic $D$-group $(G,s)$ over $K$ is an affine algebraic group with the additional structure of a $D$-variety $s:G\to \tau_\DD G$, both over $K$, such that $s$ is a group homomorphism.
\end{definition}

\begin{remark}\label{onsection}
At the level of the coordinate ring $K[G]$, a section $s:G\to \tau_\DD G$ is a group homomorphism if and only if the derivations $\DD$ on $K[G]$ commute with coproduct $\D$. Indeed, the section
$$s=(\operatorname{Id}, s_1,\dots,s_m):G\to \tau_\DD G=\tau_{\d_1} G\times_G\cdots \times_G \tau_{\d_m}G$$
is a group homomorphism iff each section 
$$(\operatorname{Id},s_i): G\to \tau_{\d_i}G,$$
for $i=1,\dots,m$, is a group homomorphism. But each such section is a group homomorphism iff $\d_i$ commutes with $\D$ (this is well known but a proof appears in \cite[Lemma 2.18]{BLM}). 
\end{remark}

If $G$ is defined over a field of constants $k$, then the zero section $s_0:G\to \tau_\DD G$ is a group homomorphism. Thus, in this case $(G,s_0)$ is a $D$-group. Our main focus here is on the (compound) isotriviality of connected commutative $D$-groups over a field of constants. It turns out that to establish compound isotriviality of such $D$-groups one essentially only needs to understand the commutative unipotent case.

\begin{lemma}\label{isoadd}
Suppose $(G,s)$ is an algebraic $D$-group over $K$. If $G=\mathbb G_a^n$ for some $n$, then $(G,s)$ is isotrivial.
\end{lemma}
\begin{proof}
By the comments in Remark \ref{onsection}, each section
$$(\operatorname{Id},s_i): \mathbb G_a^n\to \mathbb G_a^{2n}$$
is a group homomorphism. Thus, $s_i:\mathbb G_a^n\to \mathbb G_a^n$ is of the form
$s_i(\bar x)=A_i\bar x$
for some $A_i\in Mat_n(K)$, $i=1,\dots,m$. We can find $b\in \mathbb G_a^n=\mathbb G_a^n(\mathcal U)$ such that 
$$\d_i b=A_i b$$
for $i=1,\dots,m$. Set $f:\mathbb G_a^n\to \mathbb G_a^n$ be the map $f(\bar x)=\bar x- b$. It is an easy computation now to check that for every $D$-point $a$ of $\mathbb G_a^n$ we have 
$$\d_i(f(a))=0$$
Hence, $f$ is an injective $D$-morphism between $(\mathbb G_a^n,s)$ and $(\mathbb G_a^n,s_0)$ (where recall that $s_0$ is the zero section). The result follows.
 \end{proof}
 
 \begin{proposition}\label{forcom}
Suppose $(G,s)$ is a connected algebraic $D$-group over a field of constants $k$. If $G$ is commutative, then every $k$-irreducible $D$-subvariety is compound isotrivial in 2-steps.
 \end{proposition}
 \begin{proof}
 As $G$ is over a field of constants, namely $k$, it comes equipped with the zero section $s_0$ as well. Set $f(\bar x)=s(\bar x)\cdot s_0(\bar x)^{-1}$ where the product and inverse occur in $\tau_\DD G$. Then $f$ is a regular (algebraic) map from $G$ to the $m$-th power of the Lie algebra $\mathfrak L(G)$ of $G$ (here we use again that $k<\C$). Moreover, as $G$ is commutative, $f$ is group homomorphism. 
 
Let $H$ be the image of $f$; then $H=\mathbb G_a^n$ for some $n$. The section $s$ induces, via $f$, a $D$-group structure $t$ on $H$. Thus, $f$ becomes a surjective group $D$-morphism between $(G,s)$ and $(H,t)$. We claim that
$$
\xymatrix{
G\ar[r]^{ f}&H\ar[r]^{ }&0}
$$ 
witnesses the compound isotriviality of $(G,s)$. Indeed, $(H,t)$ is isotrivial by Lemma \ref{isoadd}, so it suffices to show that if $a$ is a Zariski $k$-generic $D$-point of $H$ then $f^{-1}(a)$ is isotrivial. Let $g$ be a $D$-point of $f^{-1}(a)$ and $N:=\operatorname{ker}\, f$. Then $g$ induces an injective $D$-morphism from $f^{-1}(a)$ onto $N$ (as $D$-subvarieties of $G$) given by $h \mapsto g^{-1}\cdot h$. As $N$ is defined over $k$ and $s|_N$ is the zero section, $f^{-1}(a)$ is indeed isotrivial. Note that this argument actually shows that $f^{-1}(a)$ is isotrivial for any $D$-point $a$ of $H$ (not necessarily Zariski generic).

Now let $V$ be an arbitrary $k$-irreducible $D$-subvariety of $G$. Letting $W$ be the $D$-subvariety of $H$ given by the Zariski closure of $f(V)$ (see Lemma~\ref{basic}(4)) and $g:=f|_V$, we get that
$$
\xymatrix{
V\ar[r]^{ g}&W\ar[r]^{ }&0}
$$ 
witnesses the compound isotriviality of $V$. Indeed, $(W,t|_W)$ is isotrivial because it is a $D$-subvariety of the isotrivial $(H,t)$. Also, the argument in the above paragraph shows that if $b$ is a Zariski $k$-generic $D$-point of $W$ then $f^{-1}(b)$ is isotrivial, but then $g^{-1}(b)=f^{-1}(b)\cap V$ is isotrivial as well. The result follows.
 \end{proof}
 
 \begin{remark}\label{wishfor}\
 \begin{enumerate}
 \item We note that in Proposition \ref{forcom} one cannot obtain in general compound isotriviality in 1-step (in other words, isotriviality). For example, in the single derivative case $\DD=\{\d\}$, consider $G=\mathbb G_a\times \mathbb G_m$ with $D$-group structure $s:G\to TG$ given by
 $$s(x,y)=(x,y,0,xy).$$
 Then $G$ is not isotrivial, see \cite[\S 2]{Pillay} for details. 
 \item It follows from \cite[Fact 2.7(iii)]{KowalskiPillay}, that the center $Z(G)$ of an affine algebraic $D$-group $(G,s)$ is a (normal) $D$-subgroup. In the case of a single derivation $\DD=\{\d\}$, it was shown in \cite[Theorem 2.10]{KowalskiPillay} that $G/Z(G)$ with its induced $D$-group structure is isotrivial. While this result extends to the case of several \emph{commuting} derivations, it is not yet known if it holds in the general situation of possibly noncommuting derivations. It is worth noting that if such a result does hold, then one can extend the argument of Proposition~\ref{forcom} to show that any connected algebraic $D$-group over the constants is compound isotrivial in 3-steps (this would yield an interesting extension of \cite[Proposition~2.15]{BLM}).
 \end{enumerate}
 \end{remark}
 
 Putting Proposition \ref{forcom} together with the results of Section \ref{variso}, we obtain
 
 \begin{corollary}\label{DMEcor}
 Suppose $(G,s)$ is a connected algebraic $D$-group over a field of constants $k$. If $G$ is commutative, then $(k[G],\DD)$ satisfies the $\DD$-DME.
 \end{corollary}
\begin{proof}
By Proposition~\ref{forcom}, every $k$-irreducible $D$-subvariety of $G$ is compound isotrivial. Hence, by Corollary \ref{impcor}, a prime $\DD$-ideal of $k[G]$ is $\DD$-rational only if it is $\DD$-locally closed. The result now follows from Proposition~\ref{imply}.
\end{proof}

\section{Main results on Poisson-Hopf algebras}\label{mains}

Recall that $k$ denotes a field of characteristic zero. In this section we prove the main result of the paper; namely,

\begin{theorem}\label{main}
 Any cocommutative affine Poisson-Hopf $k$-algebra satisfies the Poisson Dixmier-Moeglin equivalence. 
 \end{theorem}

 \begin{remark}\label{known}
By \cite[1.7(i), 1.10]{Oh}, in any affine Poisson algebra $(A,\{-,-\})$ we have that Poisson-locally closed implies Poisson-primitive, and Poisson-primitive implies Poisson-rational. We note that this also follows from our results in Section~\ref{variso}. Indeed, if we let $\DD$ denote the (finite) family of Hamiltonians of a collection of generators of $A$, then, by Proposition~\ref{imply}, in the differential $k$-algebra $(A,\DD)$ we have that $\DD$-locally closed implies $\DD$-primitive, and $\DD$-primitive implies $\DD$-rational. The result now follows from Corollary \ref{iff}. 
\end{remark}
 
We will make use of the following consequence of Proposition \ref{imply}.

\begin{lemma}\label{onalg}
Let $(A,\{-,-\})$ be an affine Poisson $k$-algebra such that $A\otimes k^{alg}$ is a domain. If in $A\otimes k^{alg}$ a prime ideal is Poisson-rational only if it is Poisson-locally closed, then the same holds in $A$.
\end{lemma}
\begin{proof}
The assumptions imply that $A$ is of the form $k[V]$ for some geometrically irreducible affine algebraic variety $V$ over $k$. Letting $\DD$ be the (finite) family of Hamiltonians of a collection of generators of $A$, we get that in $k^{alg}[V]$ a prime ideal is $\DD$-rational only if it is $\DD$-locally closed (by Corollary \ref{iff}). By the 'furthermore' clause of Proposition \ref{imply}, we get the same implication holds in $k[V]$. Again by Corollary \ref{iff}, we get that in  $A=k[V]$ a prime ideal is Poisson-rational only if it is Poisson-locally closed, as desired.
\end{proof} 

A commutative and cocommutative affine Hopf $k$-algebra is nothing more than the coordinate ring $k[G]$ of a connected commutative affine algebraic group $G$ over $k$. The following well known theorem characterizes such groups over $k^{alg}$ (see \cite{Borel} for instance).

\begin{theorem}\label{groups}
Let $G$ be a connected commutative affine algebraic group over $k$. Then $G$ is isomorphic over $k^{alg}$ to $\mathbb G_a^s\times \mathbb G_m^t$ for some $s$ and $t$.
\end{theorem}

A consequence of this result is that for any affine Hopf $k$-algebra $A$, that is commutative and cocommutative, we have that
\begin{equation}\label{same}
A\otimes k^{alg}=k^{alg}[x_1,\dots,x_s,y_1^{\pm},\dots,y_t^{\pm}]
\end{equation}
where the $x_i$'s are primitive and the $y_i$'s are group-like. We use this fact to prove:

\begin{proposition}\label{key}
Let $(A,\{-,-\})$ be an affine Poisson $k$-algebra. Suppose further that $k$ is algebraically closed and $A$ is a cocommutative Hopf algebra. Then, there is  $\DD \subseteq \Der_k(A)$ with
\begin{equation}\label{holds}
\operatorname{span}_A\DD=\operatorname{span}_A(\{a,-\}:\, a\in A),
\end{equation}
and such that $(A,\{-,-\})$ is a Poisson-Hopf algebra if and only if $(A,\DD)$ is a differential-Hopf algebra. 
\end{proposition}
\begin{proof}
Write $A=k[G]$ where $G$ is a connected commutative affine algebraic group over $k$. By Theorem \ref{groups} (or \eqref{same} rather), we may assume that 
$$A=k[x_1,\dots,x_s,y_1^{\pm},\dots,y_t^{\pm}]$$ where the $x_i$'s are primitive and the $y_i$'s are group-like. Consider the Hamiltonians $\d_{x_i}:=\{x_i,-\}:A\to A$, for $i=1,\dots,s$, and the normalized Hamiltonians $\d_{y_i}:=y_i^{-1}\{y_i,-\}:A\to A$, for $i=1,\dots,t$. We claim that 
$$\DD:=\{\d_{x_1},\dots,\d_{x_s}.\d_{y_1},\dots,\d_{y_t}\}$$
is the desired set of $k$-linear derivations. Clearly \eqref{holds} is satisfied (see Remark \ref{ongenset}). Now suppose that $(A,\{-,-\})$ is a Poisson-Hopf algebra. Let $1\leq i\leq s$ and set $x:=x_i$. Recall that $\D x=x\otimes 1+1\otimes x$. Now let $a\in A$, using sumless Sweedler notation we write $\D a=a_{(1)}\otimes a_{(2)}$. We then have
\begin{align*}
\d_x(\D(a)) & = \d_x(a_{(1)}\otimes a_{(2)}) \\
&= \d_x a_{(1)}\otimes a_{(2)}+a_{(1)}\otimes \d_x a_{(2)} \\
& = \{x,a_{(1)}\}\otimes a_{(2)}+ a_{(1)}\otimes \{x,a_{(2)}\} \\
&= \{x \otimes 1,a_{(1)}\otimes a_{(2)}\}+ \{1\otimes x,a_{(1)}\otimes a_{(2)}\} \\
&= \{x\otimes 1+1\otimes x, a_{(1)}\otimes a_{(2)}\} \\
&= \{\D x,\D a\} \\
&= \D(\{x,a\}) \quad\quad \quad\quad \quad \text{(since $A$ is a Poisson-Hopf algebra)} \\
&= \D(\d_x a) \\
\end{align*}

Now let $1\leq j\leq t$ and set $y=y_j$. Recall that $\D y=y\otimes y$. Now let $a\in A$. We then have
\begin{align*}
\d_y(\D(a)) & = \d_y(a_{(1)}\otimes a_{(2)}) \\
&= \d_y a_{(1)}\otimes a_{(2)}+a_{(1)}\otimes \d_y a_{(2)} \\
& = y^{-1}\{y,a_{(1)}\}\otimes a_{(2)}+ a_{(1)}\otimes y^{-1}\{y,a_{(2)}\} \\
&=\left(y^{-1}\otimes y^{-1}\right)\left(\{y,a_{(1)}\}\otimes y a_{(2)}+ y a_{(1)}\otimes \{y,a_{(2)}\}\right) \\
&= \D(y^{-1})\, \{y\otimes y, a_{(1)}\otimes a_{(2)}\} \\
&= \D(y^{-1})\, \{\D y,\D a\} \\
&= \D(y^{-1})\, \D(\{y,a\}) \quad\quad \quad\quad \quad \text{(since $A$ is a Poisson-Hopf algebra)} \\
&= \D(y^{-1}\{y,a\}) \\
&= \D(\d_y a) \\
\end{align*}
We have shown that all these derivations commute with coproduct; in other words, that $(A,\DD)$ is a differential-Hopf algebra.

The other implication (i.e., that if $\DD$ commutes with $\D$ then $\{-,-\}$ commutes with $\D$) follows from a similar series of equalities and applying Remark \ref{gen}.
\end{proof}

\begin{remark}\label{suggest}\
\begin{enumerate}
\item In terms of Poisson groups over an algebraically closed field $k$, in the sense of \cite[\S 1.3]{KS}, the above proposition shows that given a commutative affine algebraic group $G$ over $k$ equipped with a Poisson variety structure, one can find $\DD \subseteq \Der_k(k[G])$ such that \eqref{holds} holds, with $k[G]$ in place of $A$, and with the property that $G$ is a Poisson algebraic group if and only if it is an algebraic $D$-group (with respect to $\DD$). 
\item We do not know at this point whether or not the cocommutativity assumption can be removed from Proposition~\ref{key}. Nonetheless, we note that if this were the case and if algebraic $D$-groups over constants were compound isotrivial (see Remark~\ref{wishfor}(2)), then the proof below of Theorem~\ref{main} would work for any (not necessarily cocommutative) affine Poisson-Hopf $k$-algebra.
\end{enumerate}
\end{remark}

We can now prove Theorem \ref{main}.

 \begin{proof}[Proof of Theorem \ref{main}]
 By Remark \ref{known}, it suffices to show that $\DD$-rational implies $\DD$-locally closed. By Lemma \ref{onalg}, we may assume that $k$ is algebraically closed. Write $A$ as $k[G]$ where $G=\mathbb G_a^s\times \mathbb G_m^t$, and let $\DD$ be the family of $k$-linear derivations of $A$ obtained in Proposition \ref{key}. Then $(A,\DD)$ is a differential-Hopf algebra. 
 
 By Remark \ref{onsection}, the induced section $s:G\to \tau_\DD G$ is a group homomorphism; in other words, $(G,s)$ is a $D$-group. By Corollary \ref{DMEcor}, $(A, \DD)$ satisfies the $\DD$-DME; in particular, a prime $\DD$-ideal of $A$ is $\DD$-rational only if it is $\DD$-locally closed. By Corollary \ref{iff}, this in turn implies that a prime Poisson ideal of $A$ if Poisson-rational only if it is Poisson-locally closed, as desired.
 \end{proof}
 
 %\begin{question}\label{question}
%In general we know we cannot remove the Hopf algebra assumption in Theorem \ref{main} (see Remark \ref{old}), but one natural question to ask is: can we remove the cocommutative assumption? In the differential-Hopf algebra context in a \emph{single derivation} this assumption can indeed be removed, see \cite{BLM}.
%\end{question}
 
We conclude with the following application:

\begin{theorem}
If $A$ is the symmetric algebra of a finite dimensional Lie algebra $\mathfrak g$ over $k$, equipped with its natural Poisson bracket, then $A$ satisfies the Poisson Dixmier-Moeglin equivalence.
\end{theorem}

\begin{proof} We know from Section \ref{sec:Poisson-Hopf} that $A$ is a cocommutative affine Poisson-Hopf $k$-algebra, and so the result follows from Theorem~\ref{main}. 
\end{proof}

\bibliographystyle{plain}

\end{document}